\documentclass[%authoryear,preprint,review%,12pt
]{elsarticle}

\RequirePackage[abspath,realmainfile]{currfile}

\usepackage{amsmath,amssymb,amsthm,amsfonts,amsbsy}
\usepackage{graphicx}
 
\usepackage{enumerate}
\usepackage[final]
{showkeys}

\usepackage{datetime} 
\usepackage{xcolor}

\usepackage{mathrsfs}

\usepackage{hyperref}

\theoremstyle{plain} 
\newtheorem{theorem}{Theorem}%[section]
\newtheorem{corollary}[theorem]{Corollary}

{Lemma}

\newtheorem{proposition}[theorem]{Proposition}

\theoremstyle{definition} 

\theoremstyle{definition} 

\newtheorem*{ex*}{Example}
\theoremstyle{remark} 

\theoremstyle{remark} 
\newtheorem{remark}[theorem]{Remark}
\newtheorem*{remark*}{Remark}

\providecommand{\url}[1]{#1}

\renewcommand{\le}{\leqslant}
\renewcommand{\ge}{\geqslant}

\renewcommand{\P}{\operatorname{\mathsf{P}}} 
\newcommand{\E}{\operatorname{\mathsf{E}}}

\newcommand{\ii}[1]{\operatorname{\mathsf{I}}\{#1\}}

\newcommand{\R}{\mathbb{R}}

\newcommand{\al}{\alpha}

\newcommand{\de}{\delta}
\newcommand{\De}{\Delta}

\newcommand{\la}{\lambda}

\newcommand{\B}{\mathcal{B}}
\renewcommand{\B}{\mathrm{B}}

\newcommand{\sign}{\operatorname{sign}}

\newcommand{\widesim}[2][1.5]{
  \mathrel{\underset{#2}{\scalebox{#1}[1]{$\sim$}}}
}

\newcommand{\tv}{{\operatorname{\mathsf{TV}}}}
\newcommand{\K}{{\operatorname{\mathsf{K}}}}

\newcommand{\es}{\overset{\sign}{=}}

\journal{Statistics and Probability Letters}

\numberwithin{equation}{section}

\begin{document}

\begin{frontmatter}


\title{Monotonicity properties of the Poisson approximation to the binomial distribution%IP06-10-20, %IP
%IP06-10-20 with applications to statistical testing
}
%On the characteristic functions of the positive part and absolute value of a random variable
%\tnoteref{label1}
%
%\tnotetext[label1]{Supported by NSA grant H98230-12-1-0237}

%\title{}


\author{Iosif Pinelis}

\address{Department of Mathematical Sciences\\
Michigan Technological University\\
Houghton, Michigan 49931, USA\\
E-mail: ipinelis@mtu.edu}

\begin{abstract}
Certain monotonicity properties of the Poisson approximation to the binomial distribution are established. 
%IP06-10-20 
As a natural application of these results, exact (rather than approximate) tests of hypotheses on an unknown value of the parameter $p$ of the binomial distribution are presented. 
\end{abstract}


\end{frontmatter}

% \linenumbers

%\tableofcontents

\section{Introduction and 
summary}
\label{intro}
%Let $n$ be any natural number, and then let $m$ be any natural number $\le n$. Let $\la$ be any positive real number. 

For any natural number $n$ and any $p\in(0,1)$, let $X_{n,p}$ denote a random variable (r.v.) having the binomial distribution with parameters $n$ and $p$. For any positive real number $\la$, let $\Pi_\la$ denote a r.v.\ having the Poisson distribution with parameter $\la$. 

There are a large number of results on the accuracy of the Poisson approximation to the binomial distribution; see e.g.\ the survey \cite{novak-ProbSurveys}. In particular, \cite[inequality (29)]{novak-ProbSurveys} (which is based on \cite{barbour-eagleson83}) implies that 
\begin{equation}\label{eq:magic}
	d_\tv(X_{n,p},\Pi_{np})\le(1-e^{-np})p<np^2, 
\end{equation}
where $d_\tv$ is the total variation distance, defined by the formula 
\begin{equation*}
	d_\tv(X,Y):=\sup_{A\in\mathscr B(\R)}|\P(X\in A)-\P(Y\in A)|
\end{equation*}
for any r.v.'s $X$ and $Y$, 
with $\mathscr B(\R)$ denoting the Borel $\sigma$-algebra over $\R$. 
%with the supremum taken over all Borel subsets of $\R$. 

%IP06-10-20 One may note here that $d_\tv$ is only a pseudo-metric; the corresponding distance between the distributions is a true metric. The pseudo-metric 
The total variation distance $d_\tv$ has the following easy to establish but important shift property: 
\begin{equation*}
	d_\tv(X+Z,Y+Z)\le d_\tv(X,Y)
\end{equation*}
for any r.v.'s $X,Y,Z$ such that $Z$ is independent of $X$ and of $Y$. Since \break $d_\tv(X_{1,p}-X_{1,r})=|p-r|$ for $r\in(0,1)$, inequality \eqref{eq:magic}, together with the pseudo-metric and shift properties of $d_\tv$, immediately yields 
\begin{equation*}%\label{eq:magic}
	d_\tv(X_{n,p},\Pi_\la)\le d_\tv(X_{n,p},X_{n,\la/n})+d_\tv(X_{n,\la/n},\Pi_\la)
	\le|np-\la|+(1-e^{-\la})\la/n.  
\end{equation*}
So, $d_\tv(X_{n,p},\Pi_\la)\to0$ whenever $n,p,\la$ vary in any manner such that $np-\la\to0$ and $\min(\la,\la^2)=o(n)$. 

%IP06-10-20 According to \cite[formula (4.2) in Lemma~4.1]{serfling78}, 
Note that 
\begin{equation}\label{eq:la^*}
	d_\tv(X_{1,p},\Pi_{\la^\circ_p})\le \De(p):=p+(1-p)\ln(1-p)%\underset{p\downarrow0}
	\widesim{p\downarrow0} p^2/2, 
\end{equation}
where %IP06-10-20
\begin{equation}\label{eq:la0}
	\la^\circ_p:=-\ln(1-p). 
\end{equation}
Using again the pseudo-metric and shift properties of $d_\tv$, we immediately get 
\begin{equation*}
	d_\tv(X_{n,p},\Pi_{n\la^\circ_p})\le n\De(p)\widesim{p\downarrow0} np^2/2; 
\end{equation*}
cf.\ \cite[Theorem~4.1]{serfling78} and \eqref{eq:magic}. 

The following statement, describing the monotonicity pattern of $d_\tv(X_{1,p},\Pi_\la)$ in $\la$, implies that the choice $\la=\la^\circ_p$ in \eqref{eq:la^*} is optimal if $p\le1-e^{-1}$. 

\begin{proposition}\label{prop:}
For each $p\in(0,1)$, $d_\tv(X_{1,p},\Pi_\la)$ is (strictly) decreasing in $\la\in(0,\la^*_p]$ and (strictly) increasing in $\la\in[\la^*_p,\infty)$, where 
\begin{equation*}
	\la^*_p:=\min(\la^\circ_p,1)
	=\begin{cases}
	\la^\circ_p&\text{ if }p\le1-e^{-1},\\
	1&\text{ if }p\ge1-e^{-1};
	\end{cases}
\end{equation*}
hence, 
\begin{align*}
	\min_{\la>0}d_\tv(X_{1,p},\Pi_\la)=d_\tv(X_{1,p},\Pi_{\la^*_p})
	&=\min\big[d_\tv(X_{1,p},\Pi_{\la^\circ_p}),d_\tv(X_{1,p},\Pi_1)\big] \\ 
	&=\begin{cases}
	p+(1-p)\ln(1-p)&\text{ if }p\le1-e^{-1},\\
	p-e^{-1}&\text{ if }p\ge1-e^{-1}. 
	\end{cases}
\end{align*}
\end{proposition}

In view of the pseudo-metric and shift properties of $d_\tv$, Proposition~\ref{prop:} immediately yields

\begin{corollary}\label{cor:prop} 
\begin{align*}
	\min_{\la>0}d_\tv(X_{n,p},\Pi_\la)&\le d_\tv(X_{n,p},\Pi_{n\la^*_p}) \\ 
&%IP06-10-20=
\le\begin{cases}
	n\big(p+(1-p)\ln(1-p)\big)\widesim{p\downarrow0}np^2/2&\text{ if }p\le1-e^{-1},\\
	n(p-e^{-1})&\text{ if }p\ge1-e^{-1}. 
	\end{cases}
\end{align*}
\end{corollary}

Along with the total variation distance $d_\tv$, the Kolmogorov distance, defined by the formula 
\begin{equation*}
	d_\K(X,Y):=\sup_{x\in\R}|\P(X\le x)-\P(Y\le x)|, 
\end{equation*}
has been extensively studied. Clearly, $d_K\le d_\tv$. Therefore, all the upper bounds on $d_\tv(X,Y)$ hold for $d_\K(X,Y)$. 

%The following monotonicity counterparts of upper bounds on the Kolmogorov distance between binomial and Poisson distributions are the main results of this note. 
In the sequel, we always assume that 
$$m\in\{1,\dots,n\}.$$ 
We also use the notation $u\vee v:=\max(u,v)$ for real $u$ and $v$. 

%It is a textbook fact that 
%%\cite{novak-extreme}
%
%\cite{dumbgen-wellner}

\begin{theorem}\label{th:}
For any $p_n$ and $p_{n+1}$ in the interval %$(0,1)$ 
$[0,1]$ such that $p_n>p_{n+1}$, the following statements hold: 
\begin{enumerate}[(i)]
	\item If $(n+1)p_{n+1}\ge np_n$ and $m\ge1+np_n$, then \\ 
	$\P(X_{n+1,p_{n+1}}\ge m)>\P(X_{n,p_n}\ge m)$. 
	\item If $(n+1)p_{n+1}\le np_n$ and $m\le1+np_{n+1}$, then \\ 
	$\P(X_{n+1,p_{n+1}}\ge m)<\P(X_{n,p_n}\ge m)$. 
\end{enumerate} 
\end{theorem}

%IP06-10-20
As is well known, $X_{n,p}$ is stochastically monotone in $p$. Therefore, 
part (i) of Theorem~\ref{th:} immediately follows from the second inequality in \cite[Theorem~2.1]{anderson-samuels67}, 
whereas part (ii) of Theorem~\ref{th:} similarly follows from the first inequality in \cite[Theorem~2.3]{anderson-samuels67}. 
In turn, the second inequality in \cite[Theorem~2.1]{anderson-samuels67} was obtained in \cite{anderson-samuels67} as an immediate consequence of a more general result %by Hoeffding 
\cite{hoeff56}, whereas the first inequality in \cite[Theorem~2.3]{anderson-samuels67} was proved by a different method. 

In this note, we shall give a proof of Theorem~\ref{th:} by a single method, which works equally well for both parts of Theorem~\ref{th:}. 

Letting $p_n:=\la/n$ in Theorem~\ref{th:}, one immediately obtains 

\begin{corollary}\label{cor:1} \emph{(\cite[Corollary~2.1]{anderson-samuels67})} 
Take any $\la\in(0,\infty)$. 
\begin{enumerate}[(i)]
	\item If $m\ge1+\la$, then $\P(X_{n,\la/n}\ge m)$ is (strictly) increasing in natural $n\ge\la\vee m=m$ to $\P(\Pi_\la\ge m)$; in particular, it follows that 
\begin{equation}\label{eq:right}
	\P(X_{n,\la/n}\ge m)<\P(\Pi_\la\ge m)
\end{equation}
for such $\la,n,m$. 
	\item If $m\le\la$, then $\P(X_{n,\la/n}\ge m)$ is (strictly) decreasing in natural $n\ge\la\vee m=\la$ to $\P(\Pi_\la\ge m)$; in particular, it follows that 
\begin{equation}\label{eq:left}
	\P(X_{n,\la/n}\ge m)>\P(\Pi_\la\ge m)
\end{equation}
for such $\la,n,m$. 
\end{enumerate} 
\end{corollary}

%This corollary is also a slightly simplified version of \cite[Corollary~2.1]{anderson-samuels67}.

In turn, Corollary~\ref{cor:1} immediately yields the following monotonicity of concentration property. 

\begin{corollary}\label{cor:2}
Take any $\la\in(0,\infty)$. 
If natural numbers $m_1$ and $m_2$ are such that $m_1\le\la\le m_2$, then $\P(m_1\le X_{n,\la/n}\le m_2)$ is decreasing in natural $n\ge%\la
m_2+1$ to %\break 
$\P(m_1\le\Pi_\la\le m_2)$; in particular, it follows that 
\begin{equation}\label{eq:concentr}
	\P(m_1\le X_{n,\la/n}\le m_2)>\P(m_1\le\Pi_\la\le m_2)
\end{equation}
for such $\la,n,m_1,m_2$.   
\end{corollary}

Another monotonicity result is 

\begin{theorem}\label{th:2}
If $p_n=1-e^{-\la/n}$ for all natural $n$, then 
\begin{equation}\label{eq:th:2}
	\P(X_{n+1,p_{n+1}}\ge m)>\P(X_{n,p_n}\ge m)
\end{equation}
for all natural $n$ and all natural $m\in[2,n+1]$. For $m=1$, inequality \eqref{eq:th:2} turns into the equality. 
\end{theorem}

The choice $p_n=1-e^{-\la/n}$ of $p$ corresponds to \eqref{eq:la0}; cf.\ also Proposition~\ref{prop:}. 

Noting that $p_n=1-e^{-\la/n}$ implies $np_n\to\la$, we immediately have the following corollary of Theorem~\ref{th:2}:

\begin{corollary}\label{cor:3}
Take any $\la\in(0,\infty)$ and any natural $m\ge2$. If $p_n=1-e^{-\la/n}$ for all natural $n$, then 
$\P(X_{n,p_n}\ge m)$ is (strictly) increasing in natural $n\ge m%
-1$ to $\P(\Pi_\la\ge m)$; in particular, it follows that 
\begin{equation*}
	\P(X_{n,p_n}\ge m)<\P(\Pi_\la\ge m)
\end{equation*}
for such $\la,n,m$. 
\end{corollary}

It follows from Theorem~\ref{th:2} that the family $(X_{n,p_n})_{n=1}^\infty$ is stochastically monotone; more specifically, it is stochastically nondecreasing. 
A natural way to establish the stochastic monotonicity (SM) of a family of r.v.'s is to derive it from the monotone likelihood ratio property (MLR), which implies the monotone tail ratio property (MTR), which in turn implies the SM; for the discrete case, see e.g.\ Theorems~1.7(b) and 1.6 and Corollary~1.4 in \cite{keilson-sumita}. 

However, subtler tools than the MLR are needed to prove Theorems~\ref{th:} and \ref{th:2}. Indeed, the family $(X_{n,\la/n})_{n=1}^\infty$, considered in Corollary~\ref{cor:1} of Theorem~\ref{th:}, cannot have the MLR -- because then, in view of the aforementioned implications MLR$\implies$MTR$\implies$SM, inequalities \eqref{eq:right} and \eqref{eq:left} would have to go in the same direction. 

We cannot use the same kind of quick argument concerning Theorem~\ref{th:2}, because it does imply the SM of the family $(X_{n,p_n})_{n=1}^\infty$ (with $p_n=1-e^{-\la/n}$). Yet, we still have 

\begin{proposition}\label{prop:no MLR}
In general, the family $(X_{n,p_n})_{n=1}^\infty$ with $p_n=1-e^{-\la/n}$, considered in Theorem~\ref{th:2}, does not have the MLR. 
\end{proposition}

\medskip
\hrule
\bigskip

A natural application of inequalities \eqref{eq:right}, \eqref{eq:left}, and \eqref{eq:concentr} in Corollaries~\ref{cor:1} and \ref{cor:2} %IP06-10-20
(taking also into account the previously mentioned stochastic monotonicity of $X_{n,p}$ in $p$)
is to exact, conservative -- rather than approximate -- tests of hypotheses on an unknown value of the parameter $p$ of the binomial distribution: 

\begin{corollary}\label{cor:tests}
Take any natural $n$ and any $p_0\in(0,1)$. Let $\ii\cdot$ denote the indicator function. 
\begin{enumerate}[(i)]
	\item For any natural $m\ge np_0+1$ and $n\ge m$, the test $\de_+(X_{n,p}):=\ii{X_{n,p}\ge m}$ of the null hypothesis $H_0\colon p=p_0$ (or $H_0\colon p\le p_0$) versus the (right-sided) alternative $H_1\colon p>p_0$ is of level $\al_+:=\P(\Pi_{np_0}\ge m)$; that is, $\E\de_+(X_{n,p})\le\al_+$ for all $p\le p_0$.   
	\item For any natural $m\le np_0+1$, the test $\de_-(X_{n,p}):=\ii{X_{n,p}\le m}$ of the null hypothesis $H_0\colon p=p_0$ (or $H_0\colon p\ge p_0$) versus the (left-sided) alternative $H_1\colon p<p_0$ is of level $\al_-:=\P(\Pi_{np_0}\le m)$; that is, $\E\de_-(X_{n,p})\le\al_-$ for all $p\ge p_0$.   
	\item For any natural $m_1$, $m_2$, and $n$ such that $m_1\le np_0\le m_2$ and $n\ge m_2+1$, the test $\de_\pm(X_{n,p}):=1-\ii{m_1\le X_{n,p}\le m_2}$ of the null hypothesis $H_0\colon p=p_0$ versus the (two-sided) alternative $H_1\colon p\ne p_0$ is of level $\al_\pm:=1-\P(m_1\le\Pi_{np_0}\le m_2)$; that is, $\E\de_\pm(X_{n,p_0})\le\al_\pm$.   
\end{enumerate}
\end{corollary}

Corollary~\ref{cor:tests} follows immediately from  \eqref{eq:right}, \eqref{eq:left}, and \eqref{eq:concentr}, in view of the stochastic monotonicity of $X_{n,p}$ in $p$. Here one may note that parts (i) and (ii) of Corollary~\ref{cor:tests} do not immediately follow from each other, because of the absence of the required symmetry. 

%IP06-10-20 
\begin{remark}\label{rem:1} It is well known (see e.g.\ \cite[Theorem~3.4.1]{leh05}) that the test $\de_+(X_{n,p})=\ii{X_{n,p}\ge m}$ of $H_0\colon p=p_0$ (or $H_0\colon p\le p_0$) versus $H_1\colon p>p_0$ is a uniformly most powerful (UMP) test but of level $\P(X_{n,p_0}\ge m)$ rather than $\P(\Pi_{np_0}\ge m)$. The test $\de_-(X_{n,p})=\ii{X_{n,p}\le m}$ in part (ii) of Corollary~\ref{cor:tests} has the similar property. 
\end{remark}

Concerning Remark~\ref{rem:1}, Corollaries~\ref{cor:1} and \ref{cor:3}, and otherwise, one may also note the following result \cite{borisov-ruz}: for all $A\subseteq\R$ and $p\in[0,1)$, 
\begin{align*}
	\P(X_{n,p}\in A)&\le\frac{\P(\Pi_{np}\in A)}{1-p}, \\ 
%\end{equation}
\intertext{which implies}
%\begin{equation}
	\P(X_{n,p}\in A)&\ge\frac{\P(\Pi_{np}\in A)-p}{1-p},   
\end{align*}
again for all $A\subseteq\R$ and $p\in[0,1)$. 
Other bounds on the tail probabilities of $X_{n,p}$ were given e.g.\ in \cite{bahadur1960}.

%!!! tests

\section{Proofs}
\label{proofs}

\begin{proof}[Proof of Proposition~\ref{prop:}]
We have 
\begin{equation}\label{eq:d(la)}
	d(\la):=2d_\tv(X_{1,p},\Pi_\la)=|1-p-e^{-\la}|+|p-\la e^{-\la}|+1-e^{-\la}-\la e^{-\la}. 
\end{equation}
Let 
\begin{equation}\label{eq:la1}
	\la_1:=\la_1(p):=-\ln(1-p)=\la^\circ_p,  
\end{equation}
so that 
\begin{equation}\label{eq:la1?1}
	\la_1\le1\iff p\le1-e^{-1}. 
\end{equation}
Note that $\la e^{-\la}$ is continuously increasing in $\la\in(0,1]$ from $0$ to $e^{-1}$ and continuously decreasing in $\la\in[1,\infty)$ from $e^{-1}$ back to $0$. Therefore, 
\begin{equation}\label{eq:la e^-la>p}
	\la e^{-\la}>p\iff(p\le e^{-1}\ \&\ \la_2<\la<\la_3),
\end{equation}
where $\la_2=\la_2(p)$ and $\la_3=\la_3(p)$ are the unique roots $\la$ of the equation $\la e^{-\la}=p$ in the intervals $(0,1]$ and $[1,\infty)$, respectively. 

Further, for all $\la>0$ the inequality $e^\la>1+\la$ can be rewritten as \break $-\ln(1-\la e^{-\la})<\la$. Hence, 
$\la_1=-\ln(1-p)=-\ln(1-\la_2 e^{-\la_2})<\la_2$, so that for all $p\le e^{-1}$ 
\begin{equation}\label{eq:order}
	0<\la_1<\la_2\le1\le\la_3<\infty. 
\end{equation}
So, to complete the proof of Proposition~\ref{prop:}, it suffices to show that 
\begin{enumerate}[(I)]
	\item for $p\le e^{-1}$, $d(\la)$ is decreasing in $\la\in(0,\la_1]$ and increasing in $\la\in[\la_1,\la_2]$, in $\la\in[\la_2,\la_3]$, and in $\la\in[\la_3,\infty)$; 
	\item for $p\in(e^{-1},1-e^{-1}]$, $d(\la)$ is decreasing in $\la\in(0,\la_1]$ and increasing in $\la\in[\la_1,\infty)$; 
	\item for $p>1-e^{-1}$, $d(\la)$ is decreasing in $\la\in(0,1]$ and increasing in $\la\in[1,\la_1]$ and in $\la\in[\la_1,\infty)$.   
\end{enumerate}

Thus, we have to consider the following corresponding cases. 

\emph{Case \emph{I.1}: $p\le e^{-1}$ and $\la\in(0,\la_1]$.}\quad Then, by \eqref{eq:la1?1}, $\la_1\le1$ and, 
in view of \eqref{eq:d(la)}, \eqref{eq:la1}, \eqref{eq:la e^-la>p}, and \eqref{eq:order}, 
\begin{equation*}
	d(\la)=e^{-\la}-1+p+p-\la e^{-\la}+1-e^{-\la}-\la e^{-\la}
	=2p-2\la e^{-\la}, 
\end{equation*}
which is decreasing in $\la\in(0,1]$ and hence in $\la\in(0,\la_1]$. 

\emph{Case \emph{I.2}: $p\le e^{-1}$ and $\la\in[\la_1,\la_2]$.}\quad Then 
\begin{equation*}
	d(\la)=1-p-e^{-\la}+p-\la e^{-\la}+1-e^{-\la}-\la e^{-\la}
	=2\big(1-(1+\la)e^{-\la}\big), 
\end{equation*}
which is (easily seen to be) increasing in $\la\ge0$ and hence in $\la\in[\la_1,\la_2]$. 

\emph{Case \emph{I.3}: $p\le e^{-1}$ and $\la\in[\la_2,\la_3]$.}\quad Then 
\begin{equation*}
	d(\la)=1-p-e^{-\la}+\la e^{-\la}-p+1-e^{-\la}-\la e^{-\la}
	=2-2p-2e^{-\la}, 
\end{equation*}
which is increasing in $\la\ge0$ and hence in $\la\in[\la_2,\la_3]$. 

\emph{Case \emph{I.4}: $p\le e^{-1}$ and $\la\in[\la_3,\infty)$.}\quad Then 
\begin{equation*}
	d(\la)=1-p-e^{-\la}+p-\la e^{-\la}+1-e^{-\la}-\la e^{-\la}
	=2\big(1-(1+\la)e^{-\la}\big), 
\end{equation*}
the same as the expression for $d(\la)$ in Case~I.2, 
where this expression was seen to be increasing in $\la\ge0$ and hence in $\la\in[\la_3,\infty)$. 

\emph{Case \emph{II.1}: $p\in(e^{-1},1-e^{-1}]$ and $\la\in(0,\la_1]$.}\quad Then $\la_1\le1$ and 
\begin{equation*}
	d(\la)=e^{-\la}-1+p+p-\la e^{-\la}+1-e^{-\la}-\la e^{-\la}
	=2p-2\la e^{-\la}, 
\end{equation*}
which is decreasing in $\la\in(0,1]$ and hence in $\la\in(0,\la_1]$. 

\emph{Case \emph{II.2}: $p\in(e^{-1},1-e^{-1}]$ and $\la\in[\la_1,\infty)$.}\quad Then $\la_1\le1$ and 
\begin{equation*}
	d(\la)=1-p-e^{-\la}+p-\la e^{-\la}+1-e^{-\la}-\la e^{-\la}
	=2\big(1-(1+\la)e^{-\la}\big), 
\end{equation*}
the same as the expression for $d(\la)$ in Case~I.2, 
where this expression was seen to be increasing in $\la\ge0$ and hence in $\la\in[\la_1,\infty)$. 

\emph{Case \emph{III.1}: $p>1-e^{-1}$ and $\la\in(0,1]$.}\quad Then, by \eqref{eq:la1?1}, $\la_1>1$ and 
\begin{equation*}
	d(\la)=e^{-\la}-1+p+p-\la e^{-\la}+1-e^{-\la}-\la e^{-\la}
	=2p-2\la e^{-\la}, 
\end{equation*}
which is decreasing in $\la\in(0,1]$. 

\emph{Case \emph{III.2}: $p>1-e^{-1}$ and $\la\in[1,\la_1]$.}\quad Then 
\begin{equation*}
	d(\la)=e^{-\la}-1+p+p-\la e^{-\la}+1-e^{-\la}-\la e^{-\la}
	=2p-2\la e^{-\la}, 
\end{equation*}
which is increasing in $\la\ge1$ and hence in $\la\in[1,\la_1]$. 

\emph{Case \emph{III.3}: $p>1-e^{-1}$ and $\la\in[\la_1,\infty)$.}\quad Then 
\begin{equation*}
	d(\la)=1-p-e^{-\la}+p-\la e^{-\la}+1-e^{-\la}-\la e^{-\la}
	=2\big(1-(1+\la)e^{-\la}\big), 
\end{equation*}
the same as the expression for $d(\la)$ in Case~I.2, 
where this expression was seen to be increasing in $\la\ge0$ and hence in $\la\in[\la_1,\infty)$.  

The proof of Proposition~\ref{prop:} is now complete. 
\end{proof}

\begin{proof}[Proof of Theorem~\ref{th:}] %IP06-10-20 
It is well known that   
\begin{equation}\label{eq:Q_n:=}
	Q_n:=\P(X_{n,p_n}\ge m)=\frac{n!}{(m-1)!(n-m)!}\,J_n, 
\end{equation}
where 
\begin{equation}\label{eq:J_n:=}
	J_n:=\int_{1-p_n}^1 t^{n-m}(1-t)^{m-1}\,dt;  
\end{equation}
see e.g.\ \cite[formula~(3)]{bin-distr}. \big(The expression for $Q_n$ in \eqref{eq:Q_n:=} can be obtained by (say) repeated integration by parts for the integral in \eqref{eq:J_n:=}.\big) 
%, it is easy to obtain the following 
%well-known formula %for the cdf of the binomial distribution 
Therefore, 
\begin{align}
	Q_{n+1}-Q_n\es 	\De_n&:=
	(n+1)J_{n+1}-(n-m+1)J_n \label{eq:De:=} \\ 
&	=(n-m+1)I_1-(n+1)I_2, \label{eq:De}
\end{align}
where %IP06-10-20 
$A\es B$ %denotes the equality in sign
means $\sign A=\sign B$,  
\begin{equation}\label{eq:I_2:=}
	I_1:=\int_0^{1-p_n} t^{n-m}(1-t)^{m-1}\,dt,\quad\text{and}\quad
	I_2:=\int_0^{1-p_{n+1}} t^{n-m+1}(1-t)^{m-1}\,dt  
\end{equation}
%IP06-10-20
\big(in fact, $Q_{n+1}-Q_n=\binom n{m-1}\De_n$\big); 
%IP06-10-20
the equality in \eqref{eq:De} holds because $I_1+J_n=\B(n-m+1,m)$ and $I_2+J_{n+1}=\B(n-m+2,m)$, where $\B(k,m):=\break 
\int_0^1 t^{k-1}(1-t)^{m-1}\,dt=(k-1)!(m-1)!/(k+m-1)!$, so that $(n-m+1)(I_1+J_n)=(n+1)(I_2+J_{n+1})$. 
Next, 
\begin{equation}\label{eq:I_1=}
	I_1=I_{11}+I_{12}, 
\end{equation}
where 
\begin{equation}\label{eq:I_12:=}
	I_{11}:=\int_0^{1-p_n} t^{n-m+1}(1-t)^{m-1}\,dt\quad\text{and}\quad
	I_{12}:=\int_0^{1-p_n} t^{n-m}(1-t)^{m}\,dt; 
\end{equation}
this follows because the sum of the integrands in $I_{11}$ and $I_{12}$ equals the integrand in $I_1$. 
Further, integrating by parts, we see that 
\begin{equation}\label{eq:I_12=}
	(n-m+1)I_{12}=(1-p_n)^{n-m+1}p_n^m+mI_{11}. 
\end{equation}

Collecting now \eqref{eq:De}, \eqref{eq:I_1=}, \eqref{eq:I_12=}, \eqref{eq:I_12:=}, and \eqref{eq:I_2:=}, we have 
\begin{align}
	\De_n&= (1-p_n)^{n-m+1}p_n^m+(n+1)(I_{11}-I_2) \notag \\ 
&	=(1-p_n)^{n-m+1}p_n^m-(n+1)\int_{1-p_n}^{1-p_{n+1}} g(t)\,dt, \label{eq:=...-int g}
\end{align}
where $g(t):=t^{n-m+1}(1-t)^{m-1}$. 
The function $g$ is (strictly) increasing on the interval $[0,1-\frac{m-1}n]$ and decreasing on $[1-\frac{m-1}n,1]$. 
So, the condition $m\ge1+np_n$, which is equivalent to the condition $1-\frac{m-1}n\le1-p_n$, implies that %IP06-10-20 $g<g(1-p_n)=(1-p_n)^{n-m+1}p_n^{m-1}$ on the interval $(1-p_n,1-p_{n+1})$
$g(t)<g(1-p_n)=(1-p_n)^{n-m+1}p_n^{m-1}$ for $t\in(1-p_n,1-p_{n+1})$, whence, by \eqref{eq:=...-int g},  
\begin{align*}
	\De_n&> (1-p_n)^{n-m+1}p_n^m-(n+1)(p_n-p_{n+1})(1-p_n)^{n-m+1}p_n^{m-1} \\
	&\es p_n-(n+1)(p_n-p_{n+1})=(n+1)p_{n+1}-np_n. 
\end{align*}
Now part (i) of Theorem~\ref{th:} follows from the relation $\es$ in \eqref{eq:De:=} and the definition of $Q_n$ in \eqref{eq:Q_n:=}. 

The proof of part (ii) of Theorem~\ref{th:} is completed quite similarly. Here, we note that the condition $m\le1+np_{n+1}$ is equivalent to the condition $1-\frac{m-1}n\ge1-p_{n+1}$, which latter implies that %IP06-10-20 $g>g(1-p_n)=(1-p_n)^{n-m+1}p_n^{m-1}$ on the interval $(1-p_n,1-p_{n+1})$. 
$g(t)>g(1-p_n)=(1-p_n)^{n-m+1}p_n^{m-1}$ for $t\in(1-p_n,1-p_{n+1})$. 
\end{proof}

\begin{proof}[Proof of Theorem~\ref{th:2}] 
The case $m=1$ is trivial, because for $p_n=1-e^{-\la/n}$ we have $\P(X_{n,p_n}\ge 1)=1-(1-p_n)^n=1-e^{-\la}$ for all natural $n$. 

The case $m=n+1$ is also trivial. 

Suppose now that $1<m<n+1$. 
In view of the definitions of $\De_n$ and $J_n$ in \eqref{eq:De:=} and \eqref{eq:J_n:=},  for $\De_n(\la)$ denoting $\De_n$ with $p_n=1-e^{-\la/n}$, we have 
\begin{equation*}
	\De_n'(\la)%\pd{\De_n}\la
	\, \frac{e^{\la}}{(e^{\la/n}-1)^{m-1} }
	=\De_{n,1}(\la):=\Big(\frac{e^{\la/(n+1)}-1}{e^{\la/n}-1}\Big)^{m-1}-\frac{n-m+1}{n}. 
\end{equation*}
Next,
\begin{equation*}
	\frac{(e^{\la/(n+1)}-1)'_\la}{(e^{\la/n}-1)'_\la}=\frac n{n+1}\,e^{-\la/(n^2+n)}
\end{equation*}
is decreasing in $\la>0$. So, by the special-case l'Hospital-type rule for monotonicity (see e.g.\ \cite[Proposition~4.1]{pin06}), $\dfrac{e^{\la/(n+1)}-1}{e^{\la/n}-1}$ is decreasing in $\la>0$ and hence $\De_n'(\la)$ can only switch its sign from $+$ to $-$ as $\la$ is increasing from $0$ to $\infty$. So, for all real $\la>0$, 
\begin{equation}\label{eq:De ge}
	\De_n=\De_n(\la)\ge\min[\De_n(0),\De_n(\infty-)]=0, 
\end{equation}
since $\De_n(0)=0=\De_n(\infty-)$. 

Moreover, 
\begin{equation*}
	\De_{n,1}(0+)=g(m):=\Big(\frac n{n+1}\Big)^{m-1}-\frac{n-m+1}{n}>0
\end{equation*}
for $m>1$, because $g(1)=0$, $g'(1)=\frac1n-\ln(1+\frac1n)>0$, and the function $g$ is convex. 
Also, 
\begin{equation*}
	\De_{n,1}(\infty-)=-\frac{n-m+1}{n}<0
\end{equation*}
for $m\in(1,n]$. So, $\De_n(\la)$ is actually strictly increasing in $\la$ in a right neighborhood of $0$ and strictly decreasing in $\la$ in a left neighborhood of $\infty$. So, the inequality in \eqref{eq:De ge} is actually strict. 
Now \eqref{eq:th:2} follows by the $\es$ relation in \eqref{eq:De:=} and the definition of $Q_n$ in \eqref{eq:Q_n:=}. 
\end{proof}

\begin{proof}[Proof of Proposition~\ref{prop:no MLR}]
The MLR of the family $(X_{n,p_n})_{n=1}^\infty$ with $p_n=1-e^{-\la/n}$ consistent with the stochastic monotonicity \eqref{eq:th:2} means that for all natural $n$ and all integers $k$ such that $0\le k\le n-1$ we have 
\begin{equation}\label{eq:De_nk}
	\de_{n,k}:=P_{n+1,k+1}P_{n,k}-P_{n,k+1}P_{n+1,k}\ge0,
\end{equation}
where $P_{n,k}:=\P(X_{n,p_n}=k)$. 
It is not hard to see that %$\de_{n,k}$ equals 
\begin{equation*}
	\de_{n,k}\es\tilde\de_{n,k}:=-(n-k) (e^{\la/n}-e^{\la/(n+1)})+e^{\la/(n+1)}-1
\end{equation*}
%in sign. 
Letting now, for instance, $k\sim an$ and $\la\sim cn$ as $n\to\infty$, with constant  
%$a\in(0,\frac1{2e}]$ and $c\ge43/100$, \\ 
%!!!! use Lambert's function instead
$a\in(0,1)$ and $c\in(0,\infty)$, 
we see that $\tilde\de_{n,k}\to %\left(1-\left(1-\frac{1}{2 e}\right) c\right) e^c-1<0
\ell(a,c):=(a-h(c))c e^c<0$ for $a\in(0,h(c))$, where $h(c):=\frac{e^{-c}-1+c}c$, which latter is increasing in $c\in(0,\infty)$ from $0$ to $1$.  
Thus, inequality \eqref{eq:De_nk} will fail to hold when $k\sim an$, $\la\sim cn$, $c\in(0,\infty)$, $a\in(0,h(c))$, and $n$ is large enough. 
This completes the proof of Proposition~\ref{prop:no MLR}. 
\end{proof}

{\bf Acknowledgment.} Thanks are due to referees for useful suggestions and additional references, in particular to \cite{hoeff56,anderson-samuels67}.

%{\color{red}!!!!!! stopped here \today\ \currenttime }

%%%%\bibliographystyle{splncsnat}
%%%%\bibliographystyle{rss}
%%%%\bibliographystyle{abbrv}
%%%%\bibliographystyle{imsart-number}
\bibliographystyle{amsplain}
%%%%\bibliography{citations.nodoi}
%%%
\bibliography{P:/pCloudSync/mtu_pCloud_02-02-17/bib_files/citations01-09-20}

\end{document}